\documentclass[12pt]{amsart}

\input xy
\usepackage[english]{babel}
\usepackage[latin1]{inputenc}
\usepackage{graphicx}
\usepackage{amsmath}
\usepackage{amssymb}
\usepackage{amscd}
\usepackage{color}
\usepackage{oldgerm}
\usepackage{amsfonts}
\usepackage{newlfont}
\usepackage{longtable}
\usepackage{multirow}
\usepackage{xcolor}

\DeclareOldFontCommand{\rm}{\normalfont\rmfamily}{\mathrm}

\usepackage[all]{xy}

\usepackage{upref}

\def\F{\Bbb F}

\def\ad{\operatorname{ad}}

\def\Der{\operatorname{Der}}

\def\dim{\operatorname{dim}}

\def\End{\operatorname{End}}

\def\Hom{\operatorname{Hom}}
\def\Ker{\operatorname{Ker}}

\def\Id{\operatorname{Id}}

\def\Im{\operatorname{Im}}

\def\g{\frak g}
\def\gl{\frak{gl}}
\def\h{\frak h}

\theoremstyle{plain}\swapnumbers

\newtheorem{Theorem}{Theorem}[section]

\newtheorem{Prop}[Theorem]{Proposition}

\newtheorem{Cor}[Theorem]{Corollary}

\title{Invariant metrics on current Lie algebras}

\author
{R. Garc\'{\i}a-Delgado}
\address{Centro de Investigaci\'on en Matem\'aticas A. C.,  Unidad M\'erida; Yucat\'an, M\'exico, Carretera Sierra Papacal Chuburna Puerto Km 5, 97302 Sierra Papacal, Yuc.}
\email{rosendo.garciadelgado@alumnos.uaslp.edu.mx}


\date{\today}

\begin{document}

\maketitle

\begin{abstract}
In this work we state conditions for a current Lie algebra $\g \otimes \mathcal{S}$ to admit an invariant metric, where $\g$ is a quadratic Lie algebra and $\mathcal{S}$ is an associative and commutative algebra with unit. We also consider the reciprocal: if $\g \otimes \mathcal{S}$ admits an invariant metric, we state necessary and sufficient conditions for $\g$ to admit an invariant metric. In particular, we show that if $\g$ is an indecomposable quadratic Lie algebra, then $\g \otimes \mathcal{S}$ admits an invariant metric if and only if $\mathcal{S}$ also admits an invariant, symmetric and non-degenerate bilinear form. In addition, we prove a theorem similar to the double extension for $\g \otimes \mathcal{S}$, where $\g$ is an indecomposable, nilpotent and quadratic Lie algebra.
\end{abstract}

\section*{Introduction}

A \emph{quadratic Lie algebra} is a Lie algebra $\g$ over a field $\F$, with Lie bracket $[\,\cdot\,,\,\cdot\,]$, equipped with a non-degenerate, symmetric and invariant bilinear form $B:\g \times \g \to \F$. The invariant property means that $B(x,[y,z])=B([x,y],z)$, for all $x,y$ and $z$ in $\g$. A bilinear form $B$ satisfying these conditions is called an \emph{invariant metric}. An example of a quadratic Lie algebra is a semisimple Lie algebra with Killing form as its invariant metric. In the solvable case, the Killing form degenerates, so this bilinear form cannot be an invariant metric. In fact, there are solvable Lie algebras that do not admit an invariant metric. The question of whether quadratic non-semisimple Lie algebras exist has an affirmative answer. An immediate example of such structures is a reductive Lie algebra. Another more elaborated example is as follows: Let $\g$ be a Lie algebra and consider the co-adjoint representation $\ad^{\ast}:\g \to \g^{\ast}$. The vector space $\g \oplus \g^{\ast}$, with Lie bracket defined by $[x+\alpha,y+\beta]_{\g \oplus \g^{\ast}}=[x,y]+\ad^{\ast}(x)(\beta)-\ad^{\ast}(y)(\alpha)$, and the invariant metric defined by $B(x+\alpha,y+\beta)=\alpha(y)+\beta(x)$, for all $x,y$ in $\g$ and $\alpha,\beta$ in $\g^{\ast}$, make $\g \oplus \g^{\ast}$ into a quadratic non-semisimple Lie algebra, as $\g^{\ast}$ is an Abelian ideal. Thus, the class of quadratic Lie algebras contains strictly the semisimple Lie algebras. Since the semisimple Lie algebras were classified (see \cite{Hum}), the search for results to classify families of quadratic non-semisimple Lie algebras has been an active topic of research (see, for instance \cite{Bajo}, \cite{Duong} or \cite{Kath}).
\smallskip

In this work, we state results for quadratic Lie algebras constructed with the tensor product as described below. Let $\g$ be a Lie algebra and $\mathcal{S}$ be a commutative and associative algebra over $\F$. The \emph{current Lie algebra of $\g$ by $\mathcal{S}$} is the Lie algebra in $\g \otimes \mathcal{S}$, with bracket:
$$
[x \otimes s,y \otimes t]_{\g \otimes \mathcal{S}}=[x,y] \otimes s\,t,\,\text{ for all }x,y \in \g \text{ and }s,t \in \mathcal{S}.
$$
The process for extending a Lie algebra through tensor product has been applied to generalize results in Lie theory. For instance, in \cite{Nadina} is calculated the minimal dimension of a representation $\rho\!:\!\h_n \!\otimes \mathcal{S} \!\to \!\gl_{\F}(V)$, where $\h_n$ is the $2n+1$-dimensional Heisenberg Lie algebra and $\mathcal{S}$ is a quotient of the polynomial rings $\F[X]$. In \cite{Ochoa} is obtained the Levi decomposition for the Lie algebra derivations $\Der(\g \otimes \mathcal{S})$. Concerning to quadratic Lie algebras, in \cite{Zhu-Meng} is studied the case when $\g$ is a simple Lie algebra and $\g \otimes \mathcal{S}$ admits an invariant metric.
\smallskip

The aim of this work is to state results that can be applied for quadratic current Lie algebras of non-semisimple Lie algebras and to determine conditions so that a current Lie algebra admits an invariant metric. If $\g$ is a quadratic Lie algebra, in \textbf{Thm. \ref{corolario caracterizacion current}}, we state necessary and sufficient conditions for $\g \otimes \mathcal{S}$ to admit an invariant metric. We also prove that if $\g$ is an indecomposable quadratic Lie algebra, then $\g \otimes \mathcal{S}$ admits an invariant metric if and only if $\mathcal{S}$ also admits a non-degenerate, symmetric and invariant bilinear form (see \textbf{Cor. \ref{corolario current}}). This result generalizes the case when $\g$ is a simple Lie algebra (see \cite{Zhu-Meng}). In addition, we state a result similar to the double extension for $\g \otimes \mathcal{S}$, where $\g$ is an indecomposable, nilpotent and quadratic Lie algebra (see \textbf{Thm. \ref{doble extension}}). In \S 4, we assume that $\g \otimes \mathcal{S}$ admits an invariant metric and we state conditions so that $\g$ admits an invariant metric. For that, we consider a surjective linear map $F:\g \otimes \mathcal{S} \to \g$, satisfying $F([x \otimes s,y \otimes 1]_{\g \otimes \mathcal{S}})=[F(x \otimes s),y]$, for all $x,y$ in $\g$ and $s$ in $\mathcal{S}$. For each $s$ in $\mathcal{S}$ and a linear map $H:\g \to \g \otimes \mathcal{S}$ such that $F \circ H=\Id_{g}$, we set a linear map $\psi_s:\g \to \g^{\ast}$, and necessary and sufficient conditions to determine if $\psi_s$ yields an invariant metric on $\g$ (see \textbf{Thm. \ref{corolario caracterizacion}}). Finally, in \textbf{Prop. \ref{Prop ejemplo}}, we show a particular case when the existence of an invariant metric on $\g \otimes \mathcal{S}$ implies that $\g$ also admits one.
\smallskip

All algebras and vector spaces in this work are finite-dimensional over a field $\F$ of zero characteristic. 
\smallskip

Let $\g \otimes \mathcal{S}$ be the current Lie algebra of $\g$ by $\mathcal{S}$. For each $s$ in $\mathcal{S}$, let $\iota_s:\g \to \g \otimes \mathcal{S}$ be the linear map defined by $\iota_s(x)=x \otimes s$, for all $x$ in $\g$. If $\mathcal{S}$ has unit $1$, then $\iota_1$ is an injective Lie algebra homomorphism. Then we may say that the bracket in $\g \otimes \mathcal{S}$ extends the bracket in $\g$. Henceforth we shall use the same symbol for both brackets; that is, $[x \otimes s,y \otimes t]=[x,y] \otimes s\,t$ for all $x,y$ in $\g$ and $s,t$ in $\mathcal{S}$.
\smallskip

A \emph{centroid} of a Lie algebra $\g$ is a linear map $T:\g \to \g$ that commutes with the adjoint representation of $\g$, that is $T\circ \ad(x)=\ad(T(x))$, for all $x$ in $\g$. The symbol $\Gamma(\g)$ denotes the vector space generated by centroids of $\g$. 
\smallskip

We write $C(\g)=\{x \in \g\mid [x,y]=0,\text{ for all }y \in \g\}$, for the \emph{center} of the Lie algebra $\g$. 
\smallskip

For a non-degenerate, symmetric and bilinear form 
$B:\g\times\g\to\Bbb F$, let $B^{\flat}:\g\to\g^\ast$ be the map defined by
$B^{\flat}(x)\,(y)=B(x,y)$, for all $x$ and $y$ in $\g$.
Since $\dim_{\F}\g$ is assumed to be finite, one may consider the inverse map
$B^\sharp:\g^*\to\g$, of $B^{\flat}$. If $B$ is an invariant metric, then the linear map $B^{\flat}$ is an \emph{isomorphism of $\g$-modules}; that is,  $B^{\flat} \circ \ad(x)=\ad^\ast\!(x) \circ B^{\flat}$, for all $x$ in $\g$. 

\section{The Lie algebra $\g$ admits an invariant metric}

In this section, we state conditions so that $\g \otimes \mathcal{S}$ admits an invariant metric, where $\g$ is a quadratic Lie algebra. We start by setting a linear map $\Psi:\g \otimes \mathcal{S} \to \left(\g \otimes \mathcal{S} \right)^{\ast}$ and conditions for it to be a morphism of $\g \otimes \mathcal{S}$-modules.

\begin{Prop}\label{teorema caracterizacion current}{\sl
Let $\g$ be a quadratic Lie algebra over $\F$ with invariant metric $B$. Let $\mathcal{S}$ be an associative and commutative algebra with unit over $\F$. Then the following statements are equivalents:}
\smallskip

\rm{\textbf{(i)}} {\sl There exists a morphism of $\g \otimes \mathcal{S}$-modules $\Psi:\g \otimes \mathcal{S} \to \left(\g \otimes \mathcal{S}\right)^{\ast}$.}
\smallskip

\rm{\textbf{(ii)}} {\sl There exists a bilinear map $\alpha:\mathcal{S} \times \mathcal{S} \to\Gamma\left(\g\right)$, such that:}
\begin{equation}\label{iii-3}
\alpha(s,s^{\prime})([\g,\g])=\alpha(ss^{\prime},1)([\g,\g]),\quad \text{ for all } s,s^{\prime} \in\mathcal{S}.
\end{equation}

\end{Prop}
\begin{proof}
\textbf{(i) $\Rightarrow$ (ii) } Suppose that $\Psi:\g \otimes \mathcal{S} \to (\g \otimes \mathcal{S})^{\ast}$ is a morphism of $\g \otimes \mathcal{S}$-modules. For each $s$ in $\mathcal{S}$, let 
$$
\mathcal{F}(s)=B^{\sharp} \circ \iota_s^{\ast} \circ \Psi^{\ast}:\g \otimes \mathcal{S} \to \g.
$$
Then, $B^{\flat} \circ \mathcal{F}(s)=\iota_s^{\ast} \circ \Psi^{\ast}$ and $\mathcal{F}(s)^{\ast} \circ B^{\flat}=\Psi \circ \iota_s$; thus:
\begin{equation}\label{nueva referencia 1}
B(x,\mathcal{F}(s)(y \otimes s^{\prime}))\!=\!\Psi(x \otimes s)(y \otimes s^{\prime}),\text{ for all }x,y \in \g, \text{ and }s,s^{\prime} \in \mathcal{S}.
\end{equation} 
We claim that:
\begin{equation}\label{equivariancia S-generalizada}
\mathcal{F}(s)\left([x \otimes s^{\prime},y \otimes t]\right)=[\mathcal{F}(s t)(x \otimes s^{\prime}),y],
\end{equation}
for all $x,y$ in $\g$ and $s,s^{\prime},t$ in $\mathcal{S}$. Indeed, take $x,y,z$ in $\g$ and $s,s^{\prime},t$ in $\mathcal{S}$. From \eqref{nueva referencia 1} and using that $B$ is invariant we get:
\begin{equation}\label{nueva referencia 2}
\begin{split}
\Psi([z,x] \otimes s s^{\prime})(y \otimes t)&=B(z,[x,\mathcal{F}(s s^{\prime})(y \otimes t)]),\text{ and }\\
\Psi(z \otimes s)([x \otimes s^{\prime},y \otimes t])&=B(z,\mathcal{F}(s)([x \otimes s^{\prime},y \otimes t]))
\end{split}
\end{equation}
Due to $\Psi$ is a morphism of $\g \otimes \mathcal{S}$-modules, then $\Psi([z,x] \otimes s s^{\prime})(y \otimes t)=\Psi(z \otimes s)([x \otimes s^{\prime},y \otimes t])$. Thus, from \eqref{nueva referencia 2}, it follows:
\begin{equation}\label{fe2}
B(z,[x,\mathcal{F}(s s^{\prime})(y \otimes t)])=B(z,\mathcal{F}(s)([x \otimes s^{\prime},y \otimes t])).
\end{equation}
Using that $B$ is non-degenerate and that the Lie bracket $[\,\cdot\,,\,\cdot\,]$ is skew-symmetric, from \eqref{fe2} we obtain:
\begin{equation}\label{uno prima}
\mathcal{F}(s)([x \otimes s^{\prime},y \otimes t])=[x,\mathcal{F}(ss^{\prime})(y \otimes t)]=[\mathcal{F}(s t)(x \otimes s^{\prime}),y],
\end{equation}
which proves \eqref{equivariancia S-generalizada}. Using \eqref{equivariancia S-generalizada} or \eqref{uno prima}, it follows:
\begin{equation}\label{centroides y F}
\begin{split}
\mathcal{F}(s)\circ \iota_{s^{\prime}}\left([x,y]\right)&=\mathcal{F}(s)([x,y] \otimes s^{\prime})=\mathcal{F}(s)([x \otimes s^{\prime},y \otimes 1])\\
\,&=[\mathcal{F}(s) \circ \iota_{s^{\prime}}(x),y]=[x,\mathcal{F}(ss^{\prime})(y \otimes 1)].
\end{split}
\end{equation}
From \eqref{centroides y F} we deduce that $\mathcal{F}(s)\circ \iota_{s^{\prime}}$ belongs to $\Gamma(\g)$. Let $\alpha:\mathcal{S}\times \mathcal{S}\to \Gamma(\g)$ be defined by:
\begin{equation}\label{definicion de alpha}
\alpha(s,s^{\prime})=\mathcal{F}(s)\circ \iota_{s^{\prime}},\quad \text{ for all }s,s^{\prime} \in \mathcal{S}.
\end{equation}
Now we shall prove \eqref{iii-3}. Indeed, from \eqref{centroides y F} and \eqref{definicion de alpha}, we get:
$$
\aligned
\alpha(s,s^{\prime})([x,y])&=\mathcal{F}(s) \circ \iota_{s^{\prime}}([x,y])=[x,\mathcal{F}(ss^{\prime})(y \otimes 1)]\\
\,&=[x,\alpha(ss^{\prime},1)(y)]=\alpha(ss^{\prime},1)([x,y]).
\endaligned
$$
then \eqref{iii-3} holds. By the commutativity of $\mathcal{S}$, then it follows:
\begin{equation}\label{iii-3-2}
\alpha(s,s^{\prime})(x)=\alpha(ss^{\prime},1)(x)=\alpha(s^{\prime},s)(x),\,\text{ for all }x \in [\g,\g]
\end{equation}
Further, take $x$ in $[\g,\g]$, using that $\mathcal{S}$ is associative and \eqref{iii-3-2}, we obtain:
\begin{equation}\label{asociatividad de alpha}
\alpha(s,s^{\prime}t)(x)=\alpha(s(s^{\prime}t),1)(x)=\alpha((ss^{\prime})t,1)(x)=\alpha(ss^{\prime},t)(x).
\end{equation} 
\smallskip

\textbf{(ii) $\Rightarrow$ (i)} Take $s$ in $\mathcal{S}$. The bilinear map $(x,s^{\prime}) \mapsto \alpha(s,s^{\prime})(x)$, induces the linear map $\mathcal{F}(s):\g \otimes \mathcal{S} \to \g, \,x \otimes s^{\prime}\mapsto \alpha(s,s^{\prime})(x)$, for all $x$ in $\g$ and $s^{\prime}$ in $\mathcal{S}$. Using that $\mathcal{S}$ is associative and \eqref{asociatividad de alpha}, we get:
\begin{equation}\label{equivariancia S-generalizada 2}
\begin{split}
\mathcal{F}(s)([x \otimes s^{\prime},y \otimes t])&=\mathcal{F}(s)\circ \iota_{s^{\prime}t}([x,y])=\alpha(s,s^{\prime}t)([x,y])\\
\,&=\alpha(st,s^{\prime})([x,y])=[\alpha(st,s^{\prime})(x),y]\\
\,&=[\mathcal{F}(st)\circ \iota_{s^{\prime}}(x),y]=[\mathcal{F}(st)(x \otimes s^{\prime}),y],
\end{split}
\end{equation}
for all $x,y$ in $\g$ and for all $s,s^{\prime},t$ in $\mathcal{S}$. This proves \eqref{equivariancia S-generalizada} and \eqref{uno prima}. We define $\Psi:\g \otimes \mathcal{S} \to \left(\g \otimes \mathcal{S}\right)^{\ast}$ by:
\begin{equation}\label{fe33}
\Psi(x \otimes s)=B^{\flat}(x) \circ \mathcal{F}(s),\text{ for all }x \in \g \text{ and } s \in \mathcal{S}.
\end{equation}
We shall prove that $\Psi$ is a morphism of $\g \otimes \mathcal{S}$-modules. Using that $B$ is invariant and \eqref{uno prima}, we obtain:
\begin{equation}\label{fe3}
\begin{split}
& \Psi([x \otimes s,y \otimes s^{\prime}])(z \otimes t)=\Psi([x,y] \otimes s s^{\prime})(z \otimes t)\\
&=B^{\flat}([x,y])\circ \mathcal{F}(ss^{\prime})(z \otimes t)=B([x,y],\mathcal{F}(s s^{\prime})(z \otimes t))\\
\,&=B(x,[y,\mathcal{F}(s s^{\prime})(z \otimes t)])=B(x,\mathcal{F}(s)[y \otimes s^{\prime},z \otimes t])\\
\,&=B^{\flat}(s)\circ \mathcal{F}(s)([y \otimes s^{\prime},z \otimes t])=\Psi(x \otimes s)([y \otimes s^{\prime},z \otimes t]).
\end{split}
\end{equation}
Then, $\Psi$ is a morphism of $\g \otimes \mathcal{S}$-modules.
\end{proof}

We have stated conditions for there to be a morphism of $\g \otimes \mathcal{S}$-modules $\Psi:\g \otimes \mathcal{S} \to (\g \otimes \mathcal{S})^{\ast}$. In the next result, we will state conditions so that $\Psi$ yields an invariant metric on $\g \otimes \mathcal{S}$. 

\begin{Theorem}\label{corolario caracterizacion current}{\sl
Let $\g$ be a quadratic Lie algebra with invariant metric $B$, and let $S$ be an associative and commutative algebra with unit over $\F$. The current Lie algebra $\g \otimes \mathcal{S}$ admits an invariant metric if and only if there exists a bilinear map $\alpha:\mathcal{S} \times \mathcal{S} \to \Gamma(\g)$ satisfying:
\smallskip

\textbf{(i)} $\alpha(s,t)([\g,\g])=\alpha(st,1)([\g,\g])$ for all $s,t$ in $\mathcal{S}$.
\medskip

\textbf{(ii)} $\bigcap_{s \in \mathcal{S}}\Ker\left(\mathcal{F}(s)\right)=\{0\}$, where $\mathcal{F}(s):\g \otimes \mathcal{S} \to \g$, is defined by: $\mathcal{F}(s)(x \otimes t)=\alpha(s,t)(x)$, for all $x$ in $\g$ and $s,t$ in $\mathcal{S}$.
\medskip

\textbf{(iii)} $\alpha(s,t)^{\ast}=\alpha(t,s)$, where $B(\alpha(s,t)(x),y)=B(x,\alpha(s,t)^{\ast}(y))$, for all $x,y$ in $\g$ and $s,t$ in $\mathcal{S}$.}

\end{Theorem}
\begin{proof} $(\Leftarrow)$ Suppose that $\alpha$ satisfies \textbf{(i)}, \textbf{(ii)} and \textbf{(iii)}, and we will prove the existence of an invariant metric $\bar{B}$ on $\g \otimes \mathcal{S}$. Let $x,y$  be in $\g$ and $s,t$  be in $\mathcal{S}$. The bilinear map $(x,s) \mapsto B^{\flat}(x)\circ \mathcal{F}(s)$, induces the linear map $\Psi:\g \otimes \mathcal{S} \to \left(\g \otimes \mathcal{S} \right)^{\ast}$, $x \otimes s \mapsto B^{\flat}(x)\circ \mathcal{F}(s)$ (note that this is the same as the one defined in \eqref{fe33}). By \eqref{fe3} and \textbf{Prop. \ref{teorema caracterizacion current}.(i)}, $\Psi$ is a morphism of $\g \otimes \mathcal{S}$-modules. We define the bilinear form $\bar{B}$ on $\g \otimes \mathcal{S}$ by $\bar{B}(x \otimes s,y \otimes t)=B(\alpha(t,s)(x),y)$, then:
\begin{equation}\label{metrica en current Lie algebra}
\begin{split}
\bar{B}(x \otimes s,y \otimes t)&=B(\alpha(t,s)(x),y)=B\left(\mathcal{F}(t)(x \otimes s),y\right)\\
\,&=\Psi(y \otimes t)(x \otimes s).
\end{split}
\end{equation}
We claim that $\bar{B}$ is an invariant metric on $\g \otimes \mathcal{S}$. Indeed, let $x^{\prime}$ be in $\g \otimes \mathcal{S}$ such that $\bar{B}^{\flat}(x^{\prime})=0$, then $\bar{B}(x^{\prime},y \otimes s)=0$. From \eqref{metrica en current Lie algebra}, this amounts to say that $B(\mathcal{F}(s)(x^{\prime}),y)=0$. Since $B$ is non-degenerate, then $\mathcal{F}(s)(x^{\prime})=0$ which means that $x^{\prime}$ lies in $\Ker(\mathcal{F}(s))$, where $s$ is arbitrary. Then, by \textbf{(ii)}, $x^{\prime}=0$ and $\bar{B}$ is non-degenerate. In \eqref{fe3} we proved there that $\Psi$ is a morphism of $\g \otimes \mathcal{S}$-modules, then $\bar{B}$ is invariant. From $\alpha(s,t)^{\ast}=\alpha(t,s)$ and due to $B$ is symmetric, it follows:
$$
\aligned
\bar{B}(x \otimes s,y \otimes t)&=B(\mathcal{F}(t)(x \otimes s),y)=B(\alpha(t,s)(x),y)\\
\,&=B(x,\alpha(t,s)^{\ast}(y))=B(x,\alpha(s,t)(y))\\
\,&=B(\mathcal{F}(s)(y \otimes t),x)=\bar{B}(y \otimes t,x \otimes s),
\endaligned
$$
then $\bar{B}$ is symmetric and it is an invariant metric on $\g \otimes \mathcal{S}$.
\smallskip

$(\Rightarrow)$ Let $\bar{B}$ be the invariant metric on $\g \otimes \mathcal{S}$ and let $\Psi=\bar{B}^{\flat}$. Then $\Psi$ is an isomorphism of $\g \otimes \mathcal{S}$-modules, which proves \textbf{(i)} of \textbf{Prop. \ref{teorema caracterizacion current}}. 
\smallskip

For each $s \in \mathcal{S}$, let $\mathcal{F}(s):\g \otimes \mathcal{S} \to \g$ be the linear map defined by:
\begin{equation}\label{definicion de F(s)}
\mathcal{F}(s)=B^{\sharp} \circ \iota_s^{*} \circ \bar{B}^{\flat},\,\text{ then }\, B^{\flat} \circ \mathcal{F}(s)=\iota_s^{\ast} \circ \bar{B}^{\flat}.
\end{equation}
By \eqref{equivariancia S-generalizada} or \eqref{uno prima}, it follows $\mathcal{F}(s)([x \otimes s^{\prime},y \otimes t])=[\mathcal{F}(st)(x \otimes s^{\prime}),y]$, for all $x,y$ in $\g$ and $s,s^{\prime},t$ in $\mathcal{S}$. Let $\alpha(s,t):\g \to \g$ be the map defined by $\alpha(s,t)(x)=\mathcal{F}(s)(x \otimes t)$. By \eqref{centroides y F}, we have that $\alpha(s,t)$ belongs to $\Gamma(\g)$. We shall prove that $\alpha(s,t)^{\ast}=\alpha(t,s)$. Using \eqref{definicion de F(s)} we get:
\begin{equation}\label{transpuesta}
\begin{split}
B(\alpha(t,s)(x),y)&=B(\mathcal{F}(t)(x \otimes s),y)=B^{\flat}\circ \mathcal{F}(t)(x \otimes s)(y)\\
\,&=\iota^{\ast} \circ \bar{B}^{\flat}(x \otimes s)(y)=\bar{B}(x \otimes s,y \otimes t).
\end{split}
\end{equation}
Using the same arguments as in \eqref{transpuesta} and that $\bar{B}$ is symmetric, we obtain:
\begin{equation}\label{transpuesta 2}
B(\alpha(s,t)(y),x)\!=\!\bar{B}(y \otimes t,x \otimes s)\!=\!\bar{B}(x \otimes s,y \otimes t)\!=\!B(\alpha(t,s)(x),y)
\end{equation}
Since $B$ is non-degenerate, from \textbf{(iii)} and \eqref{transpuesta 2}, we get $\alpha(s,t)^{\ast}=\alpha(t,s)$. 
\smallskip

It remains to show that $\underset{s \in \mathcal{S}}{\bigcap}\Ker\left(\mathcal{F}(s)\right)=\{0\}$. Let $u$ be in $\g \otimes \mathcal{S}$ such that $\mathcal{F}(s)(u)=0$ for all $s$ in $\mathcal{S}$. By \eqref{definicion de F(s)}, $\bar{B}(v \otimes s,u)=B\left(v,\mathcal{F}(s)(u)\right)=0$ for all $v$ in $\g$ and $s$ in $\mathcal{S}$. Since $\bar{B}$ is non-degenerate, then $u=0$.
\end{proof}

The following result proves that the existence of an invariant metric on $\g \otimes \mathcal{S}$, forces $\mathcal{S}$ to admit a non-degenerate, symmetric, and invariant bilinear form, if $\g$ is a non-Abelian and indecomposable quadratic Lie algebra, over an algebraically closed field of zero characteristic.

\begin{Cor}\label{corolario current}{\sl
Let $\g$ be a non-Abelian and indecomposable quadratic Lie algebra with invariant metric $B$, over an algebraically closed field $\F$ of zero characteristic. Let $\mathcal{S}$ be a commutative and associative algebra with unit over $\F$. Then $\g \otimes \mathcal{S}$ admits an invariant metric if and only if $\mathcal{S}$ admits a non-degenerate, symmetric and invariant bilinear form.}
\end{Cor}
\begin{proof}
$(\Leftarrow)$ If $\mathcal{S}$ admits a non-degenerate, symmetric and invariant bilinear form $\gamma$, then the bilinear form $\bar{B}$ on $\g \otimes \mathcal{S}$ defined by:\newline $\bar{B}(x \otimes s,y \otimes s^{\prime})=B(x,y)\gamma(s,s^{\prime})$, for all $x,y$ in $\g$ and $s,s^{\prime}$ in $\mathcal{S}$, is an invariant metric on $\g \otimes \mathcal{S}$.
\smallskip

$(\Rightarrow)$ By Jordan-Chevalley decomposition (see \cite{Hum}, \textbf{Prop. 4.2}), the semisimple and nilpotent parts of any $B$-self adjoint centroid $T$ of $\g$, are also $B$-self adjoint centroids of $\g$. Since $\g$ is an indecomposable quadratic Lie algebra over an algebraically closed field $\mathbb{F}$ of zero characteristic, $T$ can be written as $T=\gamma \operatorname{Id}_{\g}+N$, where $\gamma$ is in $\F$, $\gamma \Id_{\g}$ is the semisimple part, and $N$ is the nilpotent part of $T$. This follows from the fact that each eigenspace corresponding to an eigenvalue of a semisimple $B$-self adjoint centroid of $\g$, is a non-degenerate ideal of $\g$.
\smallskip

Let $s$ and $t$ be in $\mathcal{S}$ and consider the bilinear map $\alpha\!:\!\mathcal{S}\! \!\times \!\mathcal{S} \!\to \! \Gamma(\g)$ of \textbf{Prop. \ref{teorema caracterizacion current}.(iii)}. As $\g \otimes \mathcal{S}$ admits an invariant metric, by \textbf{Thm. \ref{corolario caracterizacion current}.(iii)}, $\alpha(s,t)^{\ast}\!\!=\!\alpha(t,s)$, then $\bar{\alpha}(s,t)\!=\!\alpha(s,t)\!+\!\alpha(t,s)$ is a $B$-self adjoint centroid of $\g$ (where $\bar{\alpha}(s,t)\!\!=\!\bar{\alpha}(t,s)$). Thus, there are an scalar $\gamma(s,t)$ in $\F$ and a nilpotent $B$-self adjoint centroid $\sigma(s,t)$ of $\g$ such that:
\begin{equation}\label{expresion alfa}
\bar{\alpha}(s,t)=\gamma(s,t)\operatorname{Id}_{\g}+\sigma(s,t).
\end{equation}
We affirm that $\gamma$ is an invariant, symmetric and non-degenerate bilinear form in $\mathcal{S}$. To prove this, we will use that for any $T$ and $S$ in $\Gamma(\g)$, the following holds: $(T \circ S)\vert_{[\g,\g]}=(S \circ T)\vert_{[\g,\g]}$.
\smallskip 

Let $s,s^{\prime}$ and $t$ be in $\mathcal{S}$ and take a non-zero element $x$ in $[\g,\g]$. 
\smallskip

\textbf{Claim 1. $\gamma$ is bilinear.} From
$\bar{\alpha}(s+s^{\prime},t)=\gamma(s+s^{\prime},t)\Id_{\g}+\sigma(s+s^{\prime},t)$ and due to $\alpha$ is bilinear, we get:
\begin{equation}\label{oct18-2}
\begin{split}
& \left(\gamma(s+s^{\prime},t)-\gamma(s,t)-\gamma(s^{\prime},t)\right)x=\\
& \left(\sigma(s,t)+\sigma(s^{\prime},t)-\sigma(s+s^{\prime},t) \right)(x)
\end{split}
\end{equation}
Note that the linear map $\left(\sigma(s,t)+\sigma(s^{\prime},t)-\sigma(s+s^{\prime},t) \right)\vert_{[\g,\g]}$, is nilpotent, as $\sigma(s,t)\vert_{[\g,\g]}$, $\sigma(s^{\prime},s)\vert_{[\g,\g]}$ and $\sigma(s+s^{\prime},t)\vert_{[\g,\g]}$ are nilpotent and commute. Since $x$ is non zero, then, from \eqref{oct18-2} it follows: $\gamma(s+s^{\prime},t)-\gamma(s,t)-\gamma(s^{\prime},t)=0$, proving that $\gamma$ is bilinear. 
\smallskip

\textbf{Claim 2. $\gamma$ is invariant.} By \textbf{Thm. \ref{corolario caracterizacion current}.(i)}, $\alpha(s,s^{\prime}t)(x)=\alpha(s s^{\prime},t)(x)$ and by \eqref{expresion alfa} we get:
$$
\aligned
2\, \alpha(s,s^{\prime}t)(x)&=\gamma(s,s^{\prime}t)x+\sigma(s,s^{\prime}t)(x),\\
2\, \alpha(s s^{\prime},t)(x)&=\gamma(s s^{\prime},t)x+\sigma(s s^{\prime},t)(x).
\endaligned
$$
Then, $\left(\sigma(s,s^{\prime}t)-\sigma(s s^{\prime},t)\right)(x)=\left(\gamma(s s^{\prime},t)-\gamma(s,s^{\prime}t) \right)x$. Observe that $(\sigma(s,s^{\prime}t)-\sigma(s s^{\prime},t))\vert_{[\g,\g]}$ is a nilpotent linear map, then we proceed as in \textbf{Claim 1} to get that $\gamma(s,s^{\prime}t)=\gamma(s s^{\prime},t)$, thus $\gamma$ is invariant. Now we shall prove that $\gamma$ is symmetric. Since $\gamma$ is invariant and $\mathcal{S}$ is commutative, we have $\gamma(s,t)=\gamma(st,1)=\gamma(ts,1)=\gamma(t,s)$, for all $s,t$ in $\mathcal{S}$; then $\gamma$ is symmetric.
\smallskip

\textbf{Claim 3. $\gamma$ is non-degenerate.} We will use the following result: \emph{Let $V$ be a finite dimensional vector space and let $\mathcal{N}$ be a subspace of  $\End_{\F}(V)$ consisting of commutative and nilpotent maps. Then, there exists $w \neq 0$ in $V$ such that $T(w)=0$ for all $T$ in $\mathcal{N}$.} The proof of this result can be obtained by induction on $\dim_{\F}V$.
\smallskip

Suppose there exists $t^{\prime}$ in $\mathcal{S}$ such that $\gamma(s,t^{\prime})=0$ for all $s$ in $\mathcal{S}$. Let $w \neq 0$ be in $[\g,\g]$ such that $\sigma(s,t^{\prime})(w)=0$ for all $s$ in $\mathcal{S}$. By \eqref{expresion alfa} and \textbf{Prop. \ref{teorema caracterizacion current}.(ii)}, $\alpha(s,t^{\prime})(w)=0$. By \textbf{Thm. \ref{corolario caracterizacion current}} and \textbf{Thm. \ref{corolario caracterizacion current}.(ii)}, this amounts to say that $\mathcal{F}(s)(w \otimes t^{\prime})=\alpha(s,t^{\prime})(w)=0$. Thus, $w \otimes t^{\prime}$ lies in $\Ker(\mathcal{F}(s))$ for all $s \in \mathcal{S}$. Given that $\g \otimes \mathcal{S}$ admits an invariant metric, by \textbf{Thm. \ref{corolario caracterizacion current}.(ii)}, $w \otimes t^{\prime}=0$, thus $t^{\prime}=0$, and $\gamma$ is non-degenerate.
\end{proof}

\subsection{Heisenberg Lie algebra extended by a derivation}\label{Heisenberg Lie algebra}\label{Heisenberg Lie algebra extendida}

Let $\h_n=V_n \oplus \F \hslash$ be the $2n+1$-dimensional Heisenberg Lie algebra, with Lie bracket $[\,\cdot\,,\,\cdot\,]$ and $V_n$ a $2n$-dimensional vector space. It is known that there exists a non-degenerate and skew-symmetric bilinear form $\omega:V_n \times V_n \to \F$, such that $[x,y]=\omega(x,y)\hslash$, for all $x$ and $y$ in $V_n$ and $[\h_n,\F \hslash]=\{0\}$. The Heisenberg Lie algebra $\h_n$ does not admit an invariant metric. There is an extension of $\h_n$ by a derivation $D \in \Der(\h_n)$, which admits an invariant metric. The underlying vector space of this extension is $\F D \oplus \h_n$, and the Lie algebra structure is given by the semidirect product $\h_n(D)=\F D\ltimes \h_n$. Thereby, the Lie bracket $[\,\cdot,\,\cdot\,]_D$ in $\h_n(D)$ can be expressed as:
\begin{equation*}\label{corchete en Heisenberg extendido}
[a D+x+b \hslash,a^{\prime} D+x^{\prime}+b^{\prime}\hslash]_D=aD\left(x^{\prime}\right)-a^{\prime}D(x)+\omega\left(x,x^{\prime}\right)\hslash,
\end{equation*}
for all $a,a^{\prime},b,b^{\prime}$ in $\F$ and $x,x^{\prime}$ in $V_n$. In the following statements appear necessary and sufficient conditions for $\h_n(D)$ to admit an invariant metric (see \cite{Mac}):

\begin{itemize}

\item[(i)] $D(V_n) \subset V_n$ and $D(\hslash)=0$.
\smallskip

\item[(ii)] Let $\phi=D\vert_{V_n}:V_n \to V_n$. Then $\phi$ is invertible.
\smallskip

\item[(iii)] $\omega(\phi(x),y)=-\omega(x,\phi(y))$, for all $x$ and $y$ in $V_n$.
\smallskip

\item[(iv)] $C(\h_n(D))=\F \hslash$.

\end{itemize}

Any quadratic Lie algebra structure in $\h_n(D)$ is isometric to a quadratic Lie algebra structure with invariant metric $B$ on $\h_n(D)$ defined by: 
\begin{equation}\label{metrica en Heisenberg}
\begin{split}
x,y \in V_n\,\, \Rightarrow\,\, & B(x,y)=\omega\left(\phi^{-1}(x),y\right),\\
x \in V_n\,\, \Rightarrow\,\, & B(x,D)=B(x,\hslash)=0,\,\text{ and }\\
\,& B(D,D)=B(\hslash,\hslash)=0,\,\text{ and }\,B(D,\hslash)=1,\\
\end{split}
\end{equation}
(see \cite{Mac}). One can verify that $\Gamma\left(\h_n(D)\right)=\operatorname{Span}_{\F}\{\Id_{\h_n(D)},N\}$, where $N$ is the linear map in $\h_n(D)$ such that $N(D)=\hslash$ and $N^{2}=0$.
\smallskip

Let $\alpha:\mathcal{S} \times \mathcal{S} \to \Gamma(\h_n(D))$ be a bilinear map and let $\gamma,\xi:\mathcal{S}\times \mathcal{S}\to \F$ be the bilinear forms on $\mathcal{S}$ such that: 
\begin{equation}\label{alpha para Heisenberg}
\alpha(s,t)=\gamma(s,t)\Id_{\h_n(D)}+\xi(s,t)N,\,\text{ for all }s,t \in \mathcal{S}.
\end{equation}
We claim that $\alpha$ yields an invariant metric $\bar{B}$ in $\h_n(D) \otimes \mathcal{S}$, in the sense of \textbf{Thm. \ref{corolario caracterizacion current}}, if and only if $\gamma$ is non-degenerate, symmetric, and invariant form and $\xi$ is symmetric. The invariant metric $\bar{B}$ thus obtained is determined by the following assignments. Let $x,y$ be in $V_n$ and let $s,t$ be in $\mathcal{S}$, then:
\begin{equation}\label{metrica en current Heisenberg extendida}
\begin{split}
& \bar{B}(x \otimes s,y \otimes t)=B(x,y)\gamma(s,t),\\
& \bar{B}(D \otimes s,x \otimes t)=\bar{B}(\hslash \otimes s,\hslash \otimes t)=0,\\
& \bar{B}(D \otimes s,\hslash \otimes t)=\gamma(s,t),\,\text{ and }\,\bar{B}(D \otimes s,D \otimes t)=\xi(s,t).
\end{split}
\end{equation}
Indeed, suppose that $\alpha$ yields an invariant metric $\bar{B}$ on $\h_n(D) \otimes \mathcal{S}$, in the sense of \textbf{Thm. \ref{corolario caracterizacion current}}. Take a non-zero element $x$ in $\h_n$ and $s,s^{\prime},t$ in $\mathcal{S}$; since $\h_n=[\h_n(D),\h_n(D)]_D$, by \textbf{Thm. \ref{corolario caracterizacion current}.(i)}, $\alpha(ss^{\prime},t)(x)=\alpha(s,s^{\prime}t)(x)$, which, by \eqref{alpha para Heisenberg}, it is equivalent to $\gamma(s s^{\prime},t)x=\gamma(s,s^{\prime}t)x$, from where $\gamma(s s^{\prime},t)=\gamma(s,s^{\prime} t)$ and $\gamma$ is invariant. Using the same argument as in \textbf{Claim 2} of \textbf{Cor. \ref{corolario current}}, we get that $\gamma$ is symmetric.
\smallskip

Now we shall prove that $\xi$ is symmetric. Using \eqref{definicion de alpha}, \eqref{metrica en current Lie algebra}, \eqref{metrica en Heisenberg} and \eqref{alpha para Heisenberg}, we have:
\begin{equation}\label{X1}
\begin{split}
& \bar{B}(D \otimes s,D \otimes t)=B(\mathcal{F}(t)(D \otimes s),D)=B(\alpha(t,s)(D),D)\\
\,&=B(\gamma(t,s)D+\xi(t,s)N(D),D)=\xi(t,s)B(\hslash,D)=\xi(t,s).\\
\end{split}
\end{equation}
Using that $\bar{B}$ is symmetric in \eqref{X1}, we obtain that $\xi$ is symmetric.
Take $x,y$ in $V_n$ and $s,s^{\prime}$ in $\mathcal{S}$. Using the same argument as in \eqref{X1} and the fact that $N(x)=0$, we obtain: $\bar{B}(x \otimes s,y \otimes t)=B(x,y)\gamma(s,t)$.  By \eqref{metrica en Heisenberg}, we know that $B\vert_{V \times V}$ is non-degenerate, then $\gamma$ is non-degenerate. Using the same argument, we obtain the remaining cases of \eqref{metrica en current Heisenberg extendida}.
\smallskip

Now suppose that $\gamma$ is a non-degenerate, symmetric, and invariant bilinear form and that $\xi$ is symmetric and bilinear form in $\mathcal{S}$. Let $\alpha:\mathcal{S} \times \mathcal{S} \to \Gamma(\h_n(D))$ be the bilinear map defined by:
\begin{equation}\label{expresion para alpha en Heisenberg}
\alpha(s,t)=\gamma(s,t)\operatorname{Id}_{\h_n(D)}+\xi(s,t)N,\,\text{ for all }s,t \in \mathcal{S}.
\end{equation}
We shall verify that $\alpha$ satisfies the conditions given in \textbf{Thm. \ref{corolario caracterizacion current}}. Let $x$ be in $\h_n$ and $s$ and $t$ be in $\mathcal{S}$. Note that $[\h_n(D),\h_n(D)]_D=\h_n$ and $\Ker(N)=\h_n$, then $\alpha(s,t)(x)=\gamma(s,t)x$, and as $\gamma$ is invariant, then $\alpha$ satisfies \textbf{(i)} of \textbf{Thm. \ref{corolario caracterizacion current}}. 
\smallskip

To verify the conditions in \textbf{Thm. \ref{corolario caracterizacion current}.(ii)-(iii)}, first note that $\gamma$ is non-degenerate, then given a basis $\{s_1,\cdots,s_{\ell}\}$ of $\mathcal{S}$ ($\ell=\dim_{\F}(\mathcal{S})$), there exists a basis $\{s^{\prime}_1,\cdots,s^{\prime}_{\ell}\}$ of $\mathcal{S}$ such that $\gamma(s_i,s^{\prime}_j)=\delta_{ij}$, for all $1 \leq i,j \leq \ell$. Let $x^{\prime}=x_1 \otimes s_1+\ldots+x_{\ell}\otimes s_{\ell}$ be such that $\mathcal{F}(s)(x^{\prime})=0$ for all $s$ in $\mathcal{S}$; recall that $\mathcal{F}(s)(z \otimes t)=\alpha(s,t)(z)$ for all $z$ in $\g$ and $s,t$ in $\mathcal{S}$ (see \eqref{fe33}). From \eqref{expresion para alpha en Heisenberg} it follows:
\begin{equation}\label{fe5}
0=\mathcal{F}(s)(x^{\prime})=\sum_{i=1}^{\ell}\gamma(s,s_i)x_i+\sum_{i=1}^{\ell}\xi(s,s_i)N(x_i),\,\text{ for all }s \in \mathcal{S}.
\end{equation}
Letting $s=s^{\prime}_i$ in \eqref{fe5} and due to $\Im(N)=\F \hslash$, we get that for each $1 \leq i \leq \ell$, there exists $a_i$ in $\F$ such that $x_i=a_i \hslash$, whence $x^{\prime}=a_1 \hslash \otimes s_1+\ldots+a_{\ell}\hslash \otimes s_{\ell}$. Since $N(\hslash)=0$, from \eqref{fe5}, it follows $0=\mathcal{F}(s_j)(x^{\prime})=a_j \hslash$, for all $1 \leq j \leq \ell$. Thus $a_j=0$ for all $j$ and  $x^{\prime}=0$. Therefore, $\displaystyle{\bigcap_{s \in \mathcal{S}}}\Ker\left(\mathcal{F}(s)\right)=\{0\}$. Using that $\alpha$ is symmetric, we have $\alpha(s,t)^{\ast}=\alpha(t,s)$, for all $s$ and $t$ in $\mathcal{S}$. Thus, by \textbf{Thm. \ref{corolario caracterizacion current}}, $\alpha$ yields an invariant metric on $\h_n(D) \otimes \mathcal{S}$ which is determined by \eqref{metrica en current Heisenberg extendida}.

\section{A double extension theorem for quadratic current nilpotent Lie algebras}

In this section we will show a structure result for a current Lie algebra $\g \otimes \mathcal{S}$ equipped with an invariant metric, where $\g$ is an indecomposable, nilpotent and quadratic Lie algebra. To state this result, we will use Medina-Revoy theorem and the double extension for quadratic Lie algebras (see \cite{Med}). In addition, we will show how to extend the bilinear map $\alpha:\mathcal{S} \times \mathcal{S} \to \Gamma(\g)$ given in \textbf{Thm. \ref{corolario caracterizacion current}}, to another higher-dimensional quadratic Lie algebra, keeping the same conditions to obtain a higher-dimensional quadratic current Lie algebra. 
\smallskip

Let $\g \neq 0$ be an indecomposable, nilpotent and quadratic Lie algebra with invariant metric $B$, over an algebraically closed field $\F$ of zero characteristic. Let $\mathcal{A}$ be the subspace generated by those elements $x$ of $\g$ such that $N(x)=0$, for all $B$-self adjoint nilpotent centroid $N$ of $\g$. In \textbf{Claim 3}, given in the proof of \textbf{Cor. \ref{corolario current}}, we proved there that $\mathcal{A}$ is non-zero and it is straightforward to see that it is an ideal of $\g$. Since $\g$ is nilpotent then $\mathcal{A} \cap C(\g) \neq \{0\}$. Let $c$ be a non-zero element of $\mathcal{A} \cap C(\g)$. Due to $c$ belongs to $C(\g)$ and $\g$ is indecomposable, then $c$  belongs to $[\g,\g]$, and $\F c$ is contained in $(\F c)^{\perp}$. Using the Witt decomposition associated to $\F c$, $\g$ has the vector space decomposition $\g=\F d \oplus \h \oplus \F c$, where $(\F c)^{\perp}=\h \oplus \F c$, $\h$ is a non-degenerate subspace of $\g$, $\h^{\perp}=\F d \oplus \F c$, $B(d,d)=B(c,c)=0$ and $B(c,d)=1$. In addition, let $B_{\h}:\h \times \h \to \F$ be the restriction $B\vert_{\h \times \h}$ and let $[\,\cdot\,,\,\cdot\,]_{\h}:\h \times \h \to \h$ be the skew-symmetric and bilinear map arising from the projection of the Lie bracket $[\,\cdot\,,\,\cdot\,]$ onto $\h$, restricted to elements in $\h$. By Medina-Revoy theorem (see \cite{Med}), $\h$ is a quadratic Lie algebra, with Lie bracket $[\,\cdot\,,\,\cdot\,]_{\h}$ and invariant metric $B_{\h}$. In addition, there exists a $B_{\h}$-skew-symmetric derivation $D$ of $\h$ such that:
\begin{equation}\label{corchete doble extension}
[x,y]\!=\![x,y]_{\h}\!+\!B_{\h}(D(x),y)c,\,\text{ and }\,[d,x]\!=\!D(x),\,\text{ for all }x,y \in \h.
\end{equation}
We say that $\g$ is the \emph{double extension of $\h$ by $D$} and we write $\g=\h(D)$.
\smallskip

Suppose that the current Lie algebra $\g \otimes \mathcal{S}$ admits an invariant metric $\bar{B}$ and consider the bilinear map $\alpha:\mathcal{S} \times \mathcal{S} \to \Gamma(\g)$ of \textbf{Thm. \ref{corolario caracterizacion current}}, such that:
\begin{equation}\label{relacion metricas con alpha}
\bar{B}(x \otimes s,y \otimes t)=B(\alpha(t,s)(x),y)=B(x,\alpha(s,t)(y)),
\end{equation}
for all $x,y$ in $\g$ and $s,t$ in $\mathcal{S}$ (see \eqref{metrica en current Lie algebra}). Then, $\alpha(s,t)+\alpha(t,s)$ is a $B$-self adjoint centroid of $\g$. As in the proof of \textbf{Cor. \ref{corolario current}}, there are bilinear maps $\gamma:\mathcal{S} \otimes \mathcal{S} \to \F$ and $\sigma:\mathcal{S} \times \mathcal{S} \to \Gamma(\g)$, such that:
$$
\alpha(s,t)+\alpha(t,s)=\gamma(s,t)\Id_{\g}+\sigma(s,t),\quad \text{ for all }s,t \in \mathcal{S},
$$
where $\sigma(s,t)$ is a $B$-self adjoint nilpotent centroid of $\g$. The choice of $c$ implies $\sigma(s,t)(c)=0$. Then, by \textbf{Thm. \ref{corolario caracterizacion current}.(i)}, $2\alpha(s,t)(c)=\gamma(s,t)c$, which proves that $\F c$ is an \emph{ideal of $\g$ invariant under $\alpha$}, that is $\alpha(s,t)(\F c) \subset \F c$ for all $s,t$ in $\mathcal{S}$. Similarly, by \eqref{relacion metricas con alpha}, we have:
\begin{equation}\label{eq1DB}
B\left(\alpha(s,t)((\F c)^{\perp}),\F c \right)=B\left((\F c)^{\perp},\alpha(t,s)(\F c) \right)=\{0\},
\end{equation}
for all $s,t$ in $\mathcal{S}$. Then, $(\F c)^{\perp}$ is also an ideal of $\g$ invariant under $\alpha$. Observe that both $\F c \otimes \mathcal{S}$ and $(\F c)^{\perp} \otimes \mathcal{S}$ are ideals of $\g \otimes \mathcal{S}$ satisfying $\F c \otimes \mathcal{S} \subset (\F c)^{\perp} \otimes \mathcal{S}$. We claim that $(\F c)^{\perp} \otimes \mathcal{S}=(\F c \otimes \mathcal{S})^{\perp}$. Indeed, from \eqref{relacion metricas con alpha} and \eqref{eq1DB} we deduce that $(\F c)^{\perp} \otimes \mathcal{S} \subset (\F c \otimes \mathcal{S})^{\perp}$. Now we will prove the other inclusion. Let $\bar{x}$ be an element of $(\F c \otimes \mathcal{S})^{\perp}$. Since $\g=\F d \oplus \h \oplus \F c$, we can write $\bar{x}$ as: $\bar{x}=d \otimes s+\bar{y}+c \otimes t$, where $\bar{y}$ is in $\h \otimes \mathcal{S}$ and $s,t$ are elements of $\mathcal{S}$. Let $s^{\prime}$ be an arbitrary element of $\mathcal{S}$. As $\bar{y}+c \otimes t$ belongs to $(\h \otimes \mathcal{S}) \oplus (\F c \otimes \mathcal{S})$, and
$$
(\h \otimes \mathcal{S}) \oplus (\F c \otimes \mathcal{S}) \subset (\F c)^{\perp} \otimes \mathcal{S} \subset (\F c \otimes \mathcal{S})^{\perp},
$$
then $\bar{B}(\bar{x},c \otimes s^{\prime})=\bar{B}(D \otimes s,c \otimes s^{\prime})=B(D,\alpha(s,s^{\prime})(c))=0$. Since $\F c$ is an invariant ideal under $\alpha$, it follows: $B^{\flat}(\alpha(s,s^{\prime})(c))=0$, then $\alpha(s,s^{\prime})(c)=0$. Due to $c$ lies in $[\g,\g]$, by \textbf{Thm. \ref{corolario caracterizacion current}.(i)-(ii)}, we have $\mathcal{F}(s^{\prime})(c \otimes s)=\alpha(s^{\prime},s)(c)=\alpha(s,s^{\prime})(c)=0$, where $s^{\prime}$ is arbitrary, then $c \otimes s$ belongs to $\Ker(\mathcal{F}(s^{\prime}))$ for all $s^{\prime}$ in $\mathcal{S}$. Then, by \textbf{Thm. \ref{corolario caracterizacion current}.(ii)}, $c \otimes s=0$, thereby, $s=0$, and $\bar{x}=\bar{y}+c \otimes t$ belongs to $(\F c)^{\perp} \otimes \mathcal{S}$, proving that $(\F c)^{\perp} \otimes \mathcal{S}=(\F c \otimes \mathcal{S})^{\perp}$. This implies that $\g \otimes \mathcal{S}$ has the following vector space decomposition:
\begin{equation}\label{descompsicion para gS}
\g \otimes \mathcal{S}=(\F d \otimes \mathcal{S}) \oplus (\h \otimes \mathcal{S}) \oplus (\F c \otimes \mathcal{S}),
\end{equation}
where $(\F c \otimes \mathcal{S})^{\perp}=(\h \otimes \mathcal{S}) \oplus (\F c \otimes \mathcal{S})$, $\h \otimes \mathcal{S}$ is a non-degenerate subspace of $\g \otimes \mathcal{S}$ and $(\h \otimes \mathcal{S})^{\perp}=(\F d \otimes \mathcal{S}) \oplus (\F c \otimes \mathcal{S})$. According to \eqref{corchete doble extension} and \eqref{descompsicion para gS}, the Lie bracket $[\,\cdot\,,\,\cdot\,]$ of $\g \otimes \mathcal{S}$ has the following decomposition:
$$
\aligned
[x \otimes s,y \otimes t]&=[x,y]_{\h} \otimes s\,t+B_{\h}(D(x),y)\,c \otimes s\,t,\\
[D \otimes s,x \otimes t]&=D(x) \otimes s\,t,\,\text{ for all }x,y \in \h,\,\text{ and }
s,t \in \mathcal{S}.
\endaligned
$$
The choice of $c$ and \textbf{Thm. \ref{corolario caracterizacion current}.(i)} imply, $2\alpha(s,t)(c)=\gamma(s,t)c$. In the proof of \textbf{Cor. \ref{corolario current}}, we showed there that $\gamma$ is an invariant, symmetric and non-degenerate bilinear form on $\mathcal{S}$. Using \eqref{relacion metricas con alpha} and $B(d,c)=1$, we obtain, $\bar{B}(d \otimes s,c \otimes t)=B(d,\alpha(s,t)(c))=\frac{1}{2}\gamma(s,t)$, for all $s,t$ in $\mathcal{S}$. Let $\gamma^{\prime}=\frac{1}{2}\gamma$, then $\bar{B}(d \otimes s,c \otimes t)=\gamma^{\prime}(s,t)$, for all $s,t$ in $\mathcal{S}$.
\smallskip

Let $\alpha_{\h}:\mathcal{S} \times \mathcal{S} \to \End_{\F}(\h)$ be the bilinear map arising from the projection of $\alpha$ onto $\h$, restricted to $\h \times \h$. Then, for any $x$ in $\h$ and $s,t$ in $\mathcal{S}$, there exists an scalar $\beta(s,t,x)$ in $\F$, such that:
\begin{equation}\label{ecuacion para alpha 1}
\alpha(s,t)(x)=\alpha_{\h}(s,t)(x)+\beta(s,t,x)\,c \in \h \oplus \F c.
\end{equation}
Let $x,y$ be in $\h$; since $B(\alpha(s,t)(x),y)\!=\!B(x,\alpha(t,s)(y))$ (see \eqref{relacion metricas con alpha}), then 
\begin{equation}\label{transpuesta para alpha-h}
B_{\h}(\alpha_{\h}(s,t)(x),y)=B_{\h}(x,\alpha_{\h}(t,s)(y)),\,\text{ for all }s,t \in \mathcal{S}.
\end{equation}
There are bilinear maps $\eta,\zeta:\mathcal{S} \times \mathcal{S} \to \F$ and $f:\mathcal{S} \times \mathcal{S} \to \h$, such that $\alpha(s,t)(d)=\eta(s,t)d+f(s,t)+\zeta(s,t)c$, for all $s,t$ in $\mathcal{S}$. From $B(\alpha(s,t)(d),z)=B(d,\alpha(t,s)(z))$, for any $z$ in $\g$, we get: $\eta=\gamma^{\prime}$, $\beta(t,s,x)=B_{\h}(f(s,t),x)$ and $\zeta$ is symmetric; in such a way that:
\begin{equation}\label{ecuacion para D}
\alpha(s,t)(d)=\gamma^{\prime}(s,t)d+f(s,t)+\zeta(s,t)c,\,\,\text{ for all }s,t \in \mathcal{S}.
\end{equation}
We claim that:
$$
\alpha_{\h}(s,t) \text{ belongs to } \Gamma(\h), \text{ and }D \circ \alpha_{\h}(s,t)=\alpha_{\h}(s,t) \circ D=\gamma^{\prime}(s,t)D,
$$
for all $s,t$ in $\mathcal{S}$. Indeed, let $x,y$ be in $\h$; from
$$
\alpha(s,t)([d,x])=[\alpha(s,t)(d),x]=[d,\alpha(s,t)(x)],
$$
we get:
\begin{align}
\label{DE1} \alpha_{\h}(s,t)\circ D &=\gamma^{\prime}(s,t)D+\ad_{\h}(f(s,t))=D \circ \alpha_{\h}(s,t),\text{ and }\\
\label{DE2} D(f(s,t))=0,& \text{ then } f(s,t) \in \Ker(D),\,\,\text{ for all }s,t \in \mathcal{S}.
\end{align}
Analogously, from $\alpha(s,t)([x,y])=[\alpha(s,t)(x),y]=[x,\alpha(s,t)(y)]$ and using \eqref{corchete doble extension} along with \eqref{ecuacion para alpha 1}, we get $\alpha_{\h}(s,t)([x,y]_{\h})=[\alpha_{\h}(s,t)(x),y]_{\h}$, which proves that $\alpha_{\h}(s,t)$ belongs to $\Gamma(\h)$. Observe from \eqref{transpuesta para alpha-h} that: $\alpha_{\h}(s,t)^{\ast}=\alpha_{\h}(t,s)$, for all $s,t$ in $\mathcal{S}$. In addition, since $\h \otimes \mathcal{S}$ is a non-degenerate subspace of $\g \otimes \mathcal{S}$, then $\cap_{s \in \mathcal{S}}\Ker(\mathcal{F}_{\h}(s))=\{0\}$, where $\mathcal{F}_{\h}(s)(x \otimes t)=\alpha_{\h}(s,t)(x)$, which proves that $\alpha_{\h}$ satisfies the conditions of \textbf{Thm. \ref{corolario caracterizacion current}}, for the quadratic Lie algebra $(\h,[\,\cdot\,,\,\cdot\,]_{\h},B_{\h})$. Thus, the current Lie algebra $\h \otimes \mathcal{S}$ admits an invariant metric given by the restriction from $\bar{B}$ to $(\h \otimes \mathcal{S})\times (\h \otimes \mathcal{S})$.
\smallskip

Let $x,y$ be in $\h$. From $\bar{B}(d\otimes s,[x,y] \otimes s^{\prime}t)=\bar{B}([d \otimes s,x \otimes s^{\prime}],y \otimes t)$, we obtain: $\alpha_{\h}(t,ss^{\prime})\circ D=\gamma^{\prime}(s,s^{\prime}t)D$, for all $s,s^{\prime}$ and $t$ in $\mathcal{S}$. Now, we use that $\gamma^{\prime}$ is invariant to obtain $\gamma^{\prime}(s,t)D=\alpha_{\h}(s,t)\circ D$, for any $s,t$ in $\mathcal{S}$, which by \eqref{DE1}, implies that $f(s,t)$ belongs to $C(\h)$.
\smallskip

We have given conditions to extend $\alpha_{\h}:\mathcal{S} \!\times \! \mathcal{S}\! \to \!\Gamma(\h)$, from $\h$ to $\h(D)$, so that $\h(D) \otimes \mathcal{S}$ admits an invariant metric (see \eqref{ecuacion para alpha 1}, \eqref{ecuacion para D}, \eqref{DE1} and \eqref{DE2}), and we showed that any current Lie algebra $\g \otimes \mathcal{S}$ endowed with an invariant metric has this structure, where $\g=\h(D)$ is an indecomposable, nilpotent and quadratic Lie algebra. We gather these findings in the next result.

\begin{Theorem}\label{doble extension}{\sl
Let $\h$ be a quadratic Lie algebra with invariant metric $B_{\h}$. Let $D$ be a $B_{\h}$-skew symmetric derivation of $\h$ and let $c$ and $d$ be symbols. In the vector space $\h(D)=\F d \oplus \h \oplus \F c$, consider the quadratic Lie algebra structure given by the double extension of $\h$ by $D$ (see \eqref{corchete doble extension}). Let $\mathcal{S}$ be an associative and commutative algebra with unit, admitting a non-degenerate, symmetric and invariant bilinear form $\gamma^{\prime}$. Let $f:\mathcal{S} \times \mathcal{S} \to C(\h) \cap \Ker(D)$ be a bilinear map and let $\zeta:\mathcal{S} \times \mathcal{S} \to \F$ be a symmetric and bilinear form on $\mathcal{S}$. Assume that the current Lie algebra $\h \otimes \mathcal{S}$ admits an invariant metric and let $\alpha_{\h}:\mathcal{S} \times \mathcal{S} \to \Gamma(\h)$ be the bilinear map of \textbf{Thm. \ref{corolario caracterizacion current}}. If
$$
\alpha_{\h}(s,t)\circ D=D \circ \alpha_{\h}(s,t)=\gamma^{\prime}(s,t)D,\text{ for all }s,t \in \mathcal{S},
$$
then $\h(D) \otimes \mathcal{S}$ admits an invariant metric $\bar{B}$, given by:
\begin{equation}\label{bar B}
\begin{split}
\bar{B}(d \otimes s,c \otimes t)&=\gamma^{\prime}(s,t),\\
\bar{B}(x \otimes s,y \otimes t)&=B_{\h}(\alpha_{\h}(t,s)(x),y),\\
\bar{B}(d \otimes s,d \otimes t)&=\zeta(s,t),\,\text{ and }\,\bar{B}(c \otimes s,c \otimes t)=0,\\
\end{split}
\end{equation}
for all $x,y$ in $\h$ and $s,t$ in $\mathcal{S}$. In addition, any current Lie algebra $\g \otimes \mathcal{S}$ admitting an invariant metric $\bar{B}$ can be constructed in this way, where $\g=\h(D)$ is an indecomposable, quadratic and nilpotent Lie algebra over an algebraically closed field $\F$ of zero characteristic, where $\h \otimes \mathcal{S}$ admits an invariant metric.}
\end{Theorem}
\begin{proof}
The proof that $\h(D) \otimes \mathcal{S}$ admits an invariant metric if $\h \otimes \mathcal{S}$ also admits one, it can be obtained by defining $\alpha:\mathcal{S} \times \mathcal{S} \to \Gamma(\h(D))$ as follows:
$$
\aligned
\alpha(s,t)(x)&=\alpha_{\h}(s,t)(x)+B_{\h}(f(s,t),x)c,\,\text{ for all }x \in \h,\\
\alpha(s,t)(d)&=\gamma^{\prime}(s,t)\,d+f(s,t)+\zeta(s,t)c,\\
\alpha(s,t)(c)&=\gamma^{\prime}(s,t)c,\,\text{ for all }s,t \in \mathcal{S}.
\endaligned
$$
It is straightforward to verify that $\alpha$ satisfies the conditions given in \textbf{(i)}, \textbf{(ii)} and \textbf{(iii)} of \textbf{Thm. \ref{corolario caracterizacion current}}, in such a way that the invariant metric induced by $\alpha$ on $\h(D) \otimes \mathcal{S}$, is $\bar{B}$ defined in \eqref{bar B}. The proof that any current Lie algebra $\g \otimes \mathcal{S}$ endowed with an invariant metric, it can be constructed as in the statement, where $\g=\h(D)$ is an indecomposable, quadratic and nilpotent Lie algebra, is given in the analysis we made at the beginning of this section. 
\end{proof}

\section{The current Lie algebra $\g \otimes \mathcal{S}$ admits an invariant metric}

Now we assume that $\g \otimes \mathcal{S}$ is a quadratic Lie algebra, and we state conditions so that $\g$ admits an invariant metric. We start by setting a linear map $\g \to \g^{\ast}$.

\begin{Prop}\label{prop caracterizacion}{\sl
Let $\g$ be a Lie algebra and $\mathcal{S}$ be an associative and commutative algebra with unit over $\F$. Assume that the current Lie algebra $\g \otimes \mathcal{S}$ admits an invariant metric $\bar{B}$. Let $F:\g \otimes \mathcal{S} \to \g$ be a surjective linear map. For each non-zero $s$ in $\mathcal{S}$ and a linear map $H:\g \to \g \otimes \mathcal{S}$ such that $F \circ H=\operatorname{Id}_{\g}$, let $\psi_{s}$ be the linear map $\psi_s=H^{\ast} \circ \bar{B}^{\flat} \circ \iota_s:\g \to \g^{\ast}$. Then, the ordered pair $(\psi_s,F)$ makes commutative the diagram:
\begin{equation}\label{diagrama conmutativo 2}
\xymatrix{
\g \ar[r]^{\psi_s} \ar[d]_{\iota_s}& \g^{*} \ar[d]^{F^{\ast}}\\
\g \otimes \mathcal{S} \ar[r]^{{\bar{B}}^{\flat}} & \left(\g \otimes \mathcal{S} \right)^{*}
}
\end{equation} 
if and only if $\,\Im\left(\iota_s\right)^{\perp}=\Ker(F)$. If the ordered pair $(\psi_s,F)$ makes commutative the diagram \eqref{diagrama conmutativo 2}, then $\psi_s$ is bijective.}
\end{Prop}
\begin{proof}
$(\Rightarrow)$ If the ordered pair $(\psi_s,F)$ makes commutative the diagram \eqref{diagrama conmutativo 2}, then ${\bar{B}}^{\flat} \circ \iota_s=F^{\ast} \circ \psi_s$ which amounts to say:
\begin{equation}\label{comparing 1}
\bar{B}(x \otimes s,y \otimes t)=\psi_s(x)\left(F(y \otimes t)\right),\quad \mbox{ for all }x,y \in \g,\,\mbox{ and }t \in \mathcal{S}
\end{equation}
Since $\psi_s$ is equal to $H^{\ast} \circ \bar{B} \circ \iota_s$, then:
\begin{equation}\label{comparing 2}
\psi_s(x)\left(F(y \otimes t)\right)=\bar{B}\left(x \otimes s,H(F(y \otimes t))\right).
\end{equation}
Comparing \eqref{comparing 1} and \eqref{comparing 2}, we deduce that $y \otimes t-H(F(y \otimes t))$ belongs to $\Im(\iota_s)^{\perp}$, for all $y$ in $\g$ and $t$ in $\mathcal{S}$. Then $\g \otimes \mathcal{S}=\Im\left(\iota_s\right)^{\perp}+\Im(H)$ and $\Ker(F) \subset \Im(\iota_s)^{\perp}$. Let $u$ be in $\Ker(\psi_s)$, by \eqref{comparing 1}, $\bar{B}^{\flat}(u \otimes s)=0$; the fact that $\bar{B}^{\flat}$ is invertible implies $u \otimes s=0$. As $s$ is non-zero, $u=0$, then $\psi_s$ is bijective. 
\smallskip

Now we shall prove that $\Im(\iota_s)^{\perp} \subset \Ker(F)$. Let $y^{\prime}$ be in $\Im(\iota_s)^{\perp}$. As $\g \otimes \mathcal{S}=\Ker(F) \oplus \Im(H)$, there exist $w^{\prime}$ in $\Ker(F)$ and $z$ in $\g$ such that $y^{\prime}=w^{\prime}+H(z)$, then $F(y^{\prime})=z$. Consider the dual map $\psi_s^{\ast}:\g \to \g^{\ast}$, of $\psi_s$ which satisfies $\psi_s^{\ast}(x)(y)=\psi(y)(x)$, for all $x$ and $y$ in $\g$. Let $u$ be in $\g$; by \eqref{comparing 1}, $0=\bar{B}(\iota_s(u),y^{\prime})=\psi_s(u)\left(F(y^{\prime})\right)=\psi_s(u)(z)=\psi_s^{\ast}(z)(u)$. Since $u$ is arbitrary, then $\psi_s^{\ast}(z)=0$. Due to $\psi_s$ is bijective, $\psi_s^{\ast}$ is also bijective, thus $z=0$. Hence $y^{\prime}=w^{\prime}$ lies in $\Ker(F)$, and $\Im(\iota_s)^{\perp}=\Ker(F)$.
\smallskip

$(\Leftarrow)$ Suppose that $\Im\left(\iota_s\right)^{\perp}=\Ker(F)$, for some non-zero $s$ in $\mathcal{S}$. We shall prove that the ordered pair $(\psi_s,F_s)$ makes commutative the diagram \eqref{diagrama conmutativo 2}, where $\psi_s=H^{\ast} \circ \bar{B}^{\flat} \circ \iota_s$ and $F \circ H=\operatorname{Id}_{\g}$.
\smallskip

Let $x,y$ be in $\g$ and $t$ be in $\mathcal{S}$. Then $y \otimes t\!-\!H\left(F(y \otimes t)\right)$ lies in $\Ker(F)\!=\!\Im\left(\iota_s\right)^{\perp}$, and
$$
\aligned
\bar{B}(\iota_s(x),y \otimes t)&=\bar{B}\left(x \otimes s,H\left(F(y \otimes t)\right)\right)=\bar{B}^{\flat}(\iota_s(x))(H(F(y \otimes t)))\\
&=H^{\ast} \circ \bar{B}^{\flat} \circ \iota_s(x)(F(y \otimes t))=\psi_s(x)(F(y \otimes t))\\
&=F^{\ast} \circ \psi_s(x)(y \otimes t), \text{ then } \bar{B}^{\flat} \circ \iota_s=F^{\ast} \circ \psi_s.
\endaligned
$$
Therefore, $(\psi_s,F)$ makes commutative the diagram \eqref{diagrama conmutativo 2}.
\end{proof}

Now we are able to state conditions for determining if $\g$ admits an invariant metric.

\begin{Theorem}\label{corolario caracterizacion}{\sl
Let $\g$ be a Lie algebra and $\mathcal{S}$ be a commutative and associative algebra with unit over $\F$. Assume that $\g \otimes \mathcal{S}$ admits an invariant metric $\bar{B}$. Then $\g$ admits an invariant metric if and only if there are a non-zero element $s$ in $\mathcal{S}$, and an ordered pair $(\psi_s,F)$ in $\Hom_{\F}(\g,\g^{\ast}) \times \Hom_{\F}(\g \otimes \mathcal{S},\g)$, such that:} 
\smallskip

{\sl \textbf{(i)}  $(\psi_s,F)$ makes commutative the diagram \eqref{diagrama conmutativo 2}.}

{\sl \textbf{(ii)} The map $\psi_s:\!\g \to \g^{\ast}$, satisfies $\psi_s(x)(y)\!=\!\psi_s(y)(x)$, for all $x,y$ in $\g$.}

{\sl \textbf{(iii)} The map $F\!:\!\g \otimes \mathcal{S}\! \to \g$ is surjective and it satisfies: $F([x \otimes t,y \otimes 1])\!=\![F(x \otimes t),y]$, for all $x,y$ in $\g$ and $t$ in $\mathcal{S}$.}

\end{Theorem}
\begin{proof}
$(\mathbf{\Leftarrow})$ Suppose there are an element $s \neq 0$ in $\mathcal{S}$, and an ordered pair $(\psi_s,F)$ satisfying \textbf{(i),(ii)} and \textbf{(iii)}. Let $x,y,z$ be in $\g$. We claim that $B:\g \times \g \to \F$, defined by $B(x,y)=\psi_s(x)(y)$, is an invariant metric on $\g$. From \textbf{(ii)}, $B$ is symmetric. We claim that $B$ is invariant. Let $t$ be in $\mathcal{S}$. Using \textbf{(i)}-\textbf{(iii)}, and that $\bar{B}$ is invariant, we get: 
\begin{equation*}\label{psi es morfismo de modulos}
\begin{split}
& \psi_s([x,y])\left(F(z \otimes t)\right)\!=\!(F^{\ast} \circ \psi_s)([x,y])(z \otimes t)\!=\!(\bar{B}^{\flat} \circ \iota_s)([x,y])(z \otimes t)\\
&=\!\bar{B}([x,y] \otimes s,z\otimes t)\!=\!\bar{B}(x \otimes s,[y,z] \otimes t)\!=\!\bar{B}^{\flat}(\iota_s(x))([y \otimes 1,z \otimes t])\\
&=\!F^{\ast} \circ \psi_s(x)([y \otimes 1,z \otimes t])\!=\!\psi_s(x)([y,F(z \otimes t)]).\\
\end{split}
\end{equation*}
Then, $\psi_s([x,y])\left(F(z \otimes t)\right)=\psi_s(x)([y,F(z \otimes t)])$. Whence, as $F_s$ is surjective, $\psi_s([x,y])=\psi_s(x) \circ \ad(y)$, which implies that $\psi_s$ is a morphism of $\g$-modules, and $B$ is an invariant form on $\g$. As the ordered pair $(\psi_s,F)$ makes commutative the diagram \eqref{diagrama conmutativo 2}, \textbf{Prop. \ref{prop caracterizacion}} implies that $\psi_s$ is bijective; thus $B$ is an invariant metric on $\g$.
\smallskip

$(\Rightarrow)$ Suppose that the Lie algebra $\g$ admits an invariant metric $B$. We shall prove the existence of an element $s \neq 0$ in $\mathcal{S}$ and a linear map $F:\g \otimes \mathcal{S} \to \g$, such that the ordered pair $(\psi_s,F)$ makes commutative the diagram \eqref{diagrama conmutativo 2}, where $\psi_s=H^{\ast} \circ \bar{B}^{\flat} \circ \iota_s$, and $F \circ H=\operatorname{Id}_{\g}$. 
\smallskip

Let us consider a basis $\{s_1,s_2,\cdots,s_m\}$ of $\mathcal{S}$, with $s_1=1$ and let $\{p^1,p^2,\cdots,p^m\} \subset S^{\ast}$ be its dual basis. Let $\widetilde{p}:\g \otimes \mathcal{S} \to \g$ be the linear map defined by $\widetilde{p}(x \otimes s)=p^1(s)x$ for all $x$ in $\g$ and for all $s$ in $\mathcal{S}$. We define the linear maps: $F=B^{\sharp} \circ \iota_1^{\ast} \circ \bar{B}^{\flat}:\g \otimes \mathcal{S} \to \g$ and $H=\bar{B}^{\sharp} \circ {\widetilde{p}}^{*} \circ B^{\flat}:\g \to \g \otimes \mathcal{S}$.
Then, 
\begin{align}
\label{F-H} \bar{B}(x \otimes s,y \otimes 1)&=B\left(F(x \otimes s),y\right),\text{ and }\\
\label{F-H 2} \bar{B}\left(H(x),y \otimes s\right)&=p^1(s)B(x,y),\,\text{ for all }x,y \in \g\,\text{ and }s \in \mathcal{S}.
\end{align}
From \eqref{F-H}, we get $F^{*} \circ B^{\flat}=\bar{B}^{\flat} \circ \iota_1$, so $\left(B^{\flat},F\right)$ makes commutative the diagram \eqref{diagrama conmutativo 2}. In addition, from \eqref{F-H} we get $F([x \otimes s,y \otimes 1])=[F(x \otimes s),y]$, for all $x,y$ in $\g$ and $s$ in $\mathcal{S}$, thus $F$ satisfies the condition \textbf{(iii)}. Letting $s=1$ in \eqref{F-H 2}, it follows $F \circ H=\operatorname{Id}_{\g}$, and $ \bar{B}(x \otimes 1,H(y))=B(x,y)=\bar{B}(H(x),y \otimes 1)$, thus $B^{\flat}=\psi_1$ satisfies the condition \textbf{(ii)}.
\end{proof}

The following result shows a particular case when the existence of an invariant metric on $\g \otimes \mathcal{S}$ implies that $\g$ also admits one.

\begin{Prop}\label{Prop ejemplo}{\sl
Let $\g$ be a Lie algebra over $\F$ such that $\g\!\!=\!\![\g,\g]$. Let $\mathcal{S}$ be a $m$-dimensional associative and commutative algebra with unit, having a non-zero element $s$ such that $s^m\!=\!0$ and $s^{m-1} \!\neq \!0$. If $\g \otimes \mathcal{S}$ admits an invariant metric $\bar{B}$, then $\g$ also admits an invariant metric. 
}
\end{Prop}
\begin{proof}
Since $\mathcal{S}$ is $m$-dimensional and has a non-zero element $s$ such that $s^m=0$ and $s^{m-1} \neq 0$, then $\mathcal{S}=\operatorname{Span}_{\F}\{1,s,\ldots,s^{m-1}\}$. For each $t$ in $\mathcal{S}$, consider the bilinear map $B_t$ on $\g$ defined by $B_t(x,y)=\bar{B}(x \otimes t,y \otimes 1)$, for all $x$ and $y$ in $\g$. Using the skew-symmetry property of the Lie bracket, and that $\bar{B}$ is symmetric and invariant, we obtain:
\begin{equation}\label{ejemplo 13}
\begin{split}
B_t([x,y],z)\!&=\!\!\bar{B}([x \otimes 1,y \otimes t],z \otimes 1)\!=\!\!\bar{B}(x \otimes 1,[y,z] \otimes t)\\
\,&=\!\bar{B}(z \otimes t,[x,y] \otimes 1)\!=\!\!B_t(z,[x,y]),\text{ for all }x,y,z \in \!\g.
\end{split}
\end{equation}
As $\g=[\g,\g]$, we deduce from \eqref{ejemplo 13} that $B_t$ is invariant and symmetric. We claim that $B_{s^{m-1}}$ is non-degenerate. Let $k \geq 1$, using that $\bar{B}$ is invariant, we get:
\begin{equation}\label{ejemplo 1}
\bar{B}([\g,\g] \otimes s^{m-1},[\g,\g] \otimes s^{k})=\bar{B}([[\g,\g],\g] \otimes s^{m-1+k},\g \otimes 1)=\{0\},
\end{equation}
Take $x$ in $\g$ such that $B_{s^{m-1}}(x,\g)=\{0\}$. Then, by \eqref{ejemplo 1} we have: $\bar{B}(x \otimes \!s^{m-1}\!,\g \otimes s^{k})\!\!=\!\{0\}$, for all $k \geq 0$, whence $\bar{B}(x \otimes s^{m-1},\g \otimes \mathcal{S})\!=\!\!\{0\}$. Since $\bar{B}$ is non-degenerate, then $x \!\otimes \!s^{m-1}\!\!=\!0$, and $x\!\!=\!0$, thus $B_{s^{m-1}}$ is non-degenerate.
\end{proof}

\section*{Acknowledgements}

The author thanks the support 
provided by post-doctoral fellowship
CONACYT 769309. 

\bibliographystyle{amsalpha}

\end{document}